\documentclass[11pt]{amsart}%
\usepackage{hyperref}
\usepackage{amsmath}%
\usepackage{amsfonts}%
\usepackage{amssymb}%
\usepackage{amsthm}
\usepackage{mathrsfs}
\usepackage{comment}
\usepackage{todonotes}
\usepackage{bbm}
\usepackage{amscd}
\usepackage{youngtab}
\usepackage{tikz}
\usepackage{bm}
\usepackage{enumitem}
\usetikzlibrary{arrows}
\usetikzlibrary{matrix}

\newcommand{\ph}{\varphi}
\newcommand{\e}{\varepsilon}

\newcommand{\C}{\mathbb{C}}
\newcommand{\Z}{\mathbb{Z}}

\newcommand{\mf}{\mathfrak}
\newcommand{\mc}{\mathcal}

\DeclareMathOperator{\spn}{span}
\DeclareMathOperator{\snf}{snf}
\DeclareMathOperator{\code}{code}

\DeclareMathOperator{\wt}{wt}
\DeclareMathOperator{\diag}{diag}

\renewcommand{\tilde}{\widetilde}

\newtheorem{theorem}{Theorem}[section]
\newtheorem{def-prop}[theorem]{Definition-Proposition}
\newtheorem{prop}[theorem]{Proposition}

\newtheorem{lemma}[theorem]{Lemma}

\theoremstyle{definition}
\newtheorem{example}[theorem]{Example}

\theoremstyle{remark}
\newtheorem*{remark}{Remark}

\begin{document}

\title[Duality between the weak and strong orders]{A combinatorial duality between the weak and strong Bruhat orders}
\author{Christian Gaetz}
\thanks{C.G. was partially supported by an NSF Graduate Research Fellowship.}
\author{Yibo Gao}
\address{Department of Mathematics, Massachusetts Institute of Technology, Cambridge, MA.}
\email{\href{mailto:gaetz@mit.edu}{gaetz@mit.edu}} 
\email{\href{mailto:gaoyibo@mit.edu}{gaoyibo@mit.edu}} 
\date{\today}

\begin{abstract}
In recent work, the authors used an order lowering operator $\nabla$, introduced by Stanley, to prove the strong Sperner property for the weak Bruhat order on the symmetric group.  Hamaker, Pechenik, Speyer, and Weigandt interpreted $\nabla$ as a differential operator on Schubert polynomials and used this to prove a new identity for Schubert polynomials and a determinant conjecture of Stanley.  In this paper we study a raising operator $\Delta$ for the \emph{strong} Bruhat order on the symmetric group, which is in many ways dual to $\nabla$.  We prove a Schubert identity dual to that of Hamaker et al. and derive formulas for counting weighted paths in the Hasse diagrams of the strong order which agree with path counting formulas for the weak order, providing a strong order analog of Macdonald's reduced word identity.  We also show that powers of $\nabla$ and $\Delta$ have the same Smith normal forms, which we describe explicitly, answering a question of Stanley.

\smallskip
\noindent \textbf{Keywords.} weak order, Bruhat order, Schubert polynomial, duality.
\end{abstract}

\maketitle

\section{Introduction} \label{sec:intro}

The reader is referred to Section \ref{sec:background} for background and basic definitions.  Stanley \cite{Stanley2017} introduced an order lowering operator $\nabla$ for the weak (Bruhat) order $W_n$ on the symmetric group and conjectured an explicit nonvanishing formula for the determinant of 
\[
\nabla^{N - 2k}: \C (W_n)_{N-k} \to \C (W_n)_{k}.
\]
where $N={n \choose 2}$ is the rank of $W_n$.  The invertibility of this operator would imply the strong Sperner property for $W_n$, solving a problem raised by Bj\"orner \cite{Bjorner1984}.  In \cite{Gaetz2018}, the authors construct a raising operator $\Delta$, which, together with $\nabla$ determines a representation of the Lie algebra $\mathfrak{sl}_2$ on $\C W_n$, thus establishing the invertibility of $\nabla^{N - 2k}$ and the strong Sperner property.

In later work \cite{Hamaker2018}, Hamaker, Pechenik, Speyer, and Weigandt proved a new identity for derivatives of Schubert polynomials, allowing them to interpret $\nabla$ as a differential operator on the space of polynomials spanned by the Schubert polynomials $\mf{S}_w$ and thereby prove Stanley's conjecture for $\det(\nabla^{N-2k})$.

Remarkably, the operator $\Delta$ is an order raising operator supported exactly on the \emph{strong} Bruhat order $S_n$ on the symmetric group; this fact was not necessary for establishing the Sperner property of $W_n$, nor was it used for computing the determinant of $\nabla^{N-2k}$.  Our goal in this paper is to study the resulting duality between (edge-labeled versions of) $W_n$ and $S_n$.

Section \ref{sec:background} gives basic definitions and background.  Section \ref{sec:padded-Schuberts} introduces \emph{padded Schubert polynomials}: the duality between $\nabla$ and $\Delta$ takes a particularly natural form in this setting.  We deduce an identity of Schubert polynomials dual to that of Hamaker, Pechenik, Speyer, and Weigandt and use this in Section \ref{sec:path-counting} to prove weighted path-counting identities for $W_n$ and $S_n$, implying a strong order analog of Macdonald's reduced word formula \cite{Macdonald}.  These identities look similar to an identity for the \emph{Chevalley edge weights} on $S_n$ previously studied by Stembridge \cite{Stembridge2002} and Postnikov and Stanley \cite{Postnikov2009}.  In Section \ref{sec:smith-forms} we show that powers of $\nabla$ and $\Delta$ have the same Smith normal forms, which we describe in a simple way, answering a question of Stanley \cite{Stanley2017}.  This indicates that $\nabla$ and $\Delta$ determine a duality between $W_n$ and $S_n$ which is stronger than the previously observed path-counting coincidences between the $\nabla$-edge weights and the Chevalley weights.

\section{Background and definitions} \label{sec:background}
We refer the reader to \cite{Stanley2012} for basic definitions about posets in what follows.

\subsection{Order operators and edge labels} \label{sec:order-operators}
For $P$ a finite graded poset, we let $P_k$ denote the set of elements of rank $k$; we always make the convention that $P_{-1}=\emptyset$, so that $\#P_{-1}=0$ is well-defined.  For $S \subset P$, we let $\C S$ denote the vector space of formal linear combinations of elements of $S$.  A linear operator $U_k: \C P_k \to \C P_{k+1}$ (resp. $D_k: \C P_k \to \C P_{k-1}$) is called a \emph{raising operator} (resp. \emph{lowering operator}).  A raising (resp. lowering) operator is an \emph{order raising} (resp. \emph{order lowering}) operator if, when we write $U_k x = \sum_{y \in P_{k+1}} c_y y$ (or $D_k y = \sum_{x \in P_{k-1}} c_x x$) we have $c_y=0$ (or $c_x=0$) unless $x \lessdot y$.  When we have a family of such operators indexed by the rank $k$, we omit the subscripts of the operators when no confusion can result.  Given a family $U_i$ of raising operators, for $i<j$ we define
\[
U^{[i,j]}:=U_{j-1}U_{j-2}\cdots U_{i+1}U_i: \C P_i \to \C P_j,
\]
and similarly define $D^{[i,j]}:\C P_j\rightarrow\C P_i$ for lowering operators $D$.

It is clear from the definitions that an order operator $\ph$ carries the same information as a weighting of the edges in the Hasse diagram of $P$ by complex numbers.  We let $\wt_{\ph}$ denote the corresponding weight function on cover relations, and we freely move between these two forms.

Given a saturated chain $\mf{c}$ from $x \in P_i$ to $y \in P_j$, we let the weight of $\mf{c}$ be the product of the weights of the cover relations $c$ which make up the chain:
\[
\wt_{\ph}(\mf{c})=\prod_{c \in \mf{c}} \wt_{\ph}(c),
\]
and we let
\[
m_{\ph}(x,y)= \sum_{\mf{c}: x \to y} \wt(\mf{c})
\]
denote the number of weighted paths from $x$ to $y$ in the Hasse diagram of $P$, viewed as a directed graph.  It is clear that, using the natural basis $P$ for $\C P$, the matrix of $\ph^{[i,j]}$ is given by $\left( m_{\ph}(x,y) \right)_{x \in P_i, y \in P_j}$.

\subsection{The weak and strong Bruhat orders}
\label{sec:bruhat-orders}

The weak and strong (Bruhat) orders on the symmetric group $S_n$ arise from the realization of $S_n$ as a Coxeter group, and are integral to representation theory and geometry in ``type A."  

For $i=1,...,n-1$ let $s_i=(i, i+1)$ denote the simple transpositions in $S_n$, and for $1 \leq i < j \leq n$ let $t_{ij}=(i,j)$ denote general transpositions.  For $w \in S_n$, the \emph{length} $\ell(w)$ is defined to be the smallest $k$ such that $w=s_{i_1}s_{i_2} \cdots s_{i_k}$ for some choice of $i_1,...,i_k \in \{1,...,n-1\}$.  The weak order $(W_n, \leq_W)$ and strong order $(S_n, \leq_S)$ are defined by their covering relations:
\begin{itemize}
    \item $w \lessdot_W u$ if and only if $u=ws_i$ and  $\ell(u)=\ell(w)+1$,
    \item $w \lessdot_S u$ if and only if $u=wt_{ij}$ and $\ell(u)=\ell(w)+1$.
\end{itemize}
Thus the weak and strong orders share the same ground set (the symmetric group on $n$ letters) and rank structure: $(W_n)_i=(S_n)_i=\{w \in S_n \: | \: \ell(w)=i\}$.  Since the ground sets are the same, no confusion should result from our use of $S_n$ to denote both the symmetric group and the strong order.  Each order has as its unique maximal element the permutation $w_0$ with one-line notation $n,n-1,...,1$, the unique element of rank $N={n \choose 2}$, and as its unique minimal element the identity permutation $\e=1,2,...,n$, the unique permutation of length zero.

\subsection{The Smith normal form of an integer matrix}
\label{sec:def-of-SNF}
Let $A$ be an $n \times m$ integer matrix.  The \emph{Smith normal form} $B$ of $A$ is the unique $n \times m$ integer matrix with nonnegative diagonal entries $b_1,...,b_{\min(n,m)}$ such that $b_i$ divides $b_{i+1}$ for all $i$, all off-diagonal entries are 0, and $B=PAQ$ for some matrices $P \in GL_n(\Z)$ and $Q \in GL_m(\Z)$.  The action of $P$ and $Q$ can be interpreted as integer row and column operations on $A$.  We write $B=\snf(A)$.  It is clear that $\snf(A)$ does not depend on the ordering of the rows and columns of $A$, since row and column swaps are integer row and column operations.  Finally, we write $\tilde{\snf}(A)$ for the square matrix $\diag(b_1,...,b_{\min(n,m)})$ which is the Smith normal form of $A$ with extra rows and columns of zeros removed.

Since elements of $GL_n(\Z)$ have determinant $\pm 1$, if $A$ is a square matrix we have:
\[
|\det(A)|=\det(\snf(A))=\prod_i b_i.
\]
Thus the Smith normal form is a considerable refinement of the absolute value of the determinant of a square integer matrix, and a generalization to rectangular matrices.  For a survey on Smith normal forms in combinatorics, see Stanley \cite{Stanley2016}.

\section{The action of $\nabla$ and $\Delta$ on padded Schubert polynomials}
\label{sec:padded-Schuberts}
For $\alpha=(\alpha_1,...,\alpha_{n-1})$ a composition of $k$, we write $x^{\alpha}$ for the monomial $\prod_{i=1}^{n-1} x_i^{\alpha_i}$, and we let $|\alpha|:=k$.  We write $\rho$ for the staircase composition $(n-1,n-2,...,2,1)$ of $N:={n \choose 2}$.  When each part of $\alpha$ is at most the corresponding part of $\rho$, we write $\alpha \leq \rho$ and we let $\rho-\alpha$ denote the composition $(n-1-\alpha_1,n-2-\alpha_2,...,1-\alpha_{n-1})$ of $N-|\alpha|$.

The \emph{Schubert polynomials} $\mf{S}_w(x_1,...,x_{n-1})$ for $w \in S_n$ form a basis for the space $V=\spn_{\C} \{x^{\alpha} \: | \: \alpha \leq \rho \}$.  They can be defined recursively as follows:
\begin{itemize}
    \item $\mf{S}_{w_0}(x_1,...,x_{n-1})=x_1^{n-1}x_2^{n-2}\cdots x_{n-1}$, and 
    \item $\mf{S}_{s_i w}=N_i \cdot \mf{S}_w$ if $\ell(s_iw)<\ell(w)$,
\end{itemize}
where $N_i$ denotes the $i$-th \emph{Newton divided difference operator}:
\[
N_i \cdot g(x_1,...,x_{n-1}) := \frac{g-s_i \cdot g}{x_i-x_{i+1}}.
\]
Here the simple transposition $(i \: i+1)$ acts on the polynomial $g$ by interchanging the variables $x_i$ and $x_{i+1}$.

We define the \emph{padded Schubert polynomials} $\tilde{\mf{S}}_w(x_1,...,x_{n-1}; y_1,...,y_{n-1})$, a basis for $\tilde{V}=\spn_{\C} \{x^{\alpha}y^{\rho-\alpha} \: | \: \alpha \leq \rho \}$, defined as the images of the $\mf{S}_w$ under the natural isomorphism $V \to \tilde{V}$ given by 
\[
x^{\alpha} \mapsto x^{\alpha}y^{\rho - \alpha}.
\]
Define operators $\nabla$ and $\Delta$ on $\tilde{V}$ by:
\begin{align*}
    \nabla &= \sum_{i=1}^{n-1} y_i \frac{\partial}{\partial x_i} \\
    \Delta &= \sum_{i=1}^{n-1} x_i \frac{\partial}{\partial y_i}.
\end{align*}

The following theorem was proved in the context of Schubert polynomials by Hamaker, Pechenik, Speyer, and Weigandt \cite{Hamaker2018}; it's extension to padded Schubert polynomials is immediate.

\begin{theorem} \label{thm:nabla}
\begin{equation} \label{eq:nabla}
\nabla \cdot \tilde{\mf{S}}_w = \sum_{ws_i \lessdot_W w} \tilde{\mf{S}}_{ws_i}.
\end{equation}
\end{theorem}

Theorem \ref{thm:nabla} implies in particular that, identifying $\C W_n$ and $\tilde{V}$ by the map $w \mapsto \tilde{\mf{S}}_w$, $\nabla$ is an order lowering operator for $W_n$.  Dually, Theorem \ref{thm:delta} shows that $\Delta$ is an order raising operator for $S_n$.  The weights are defined in terms of the \emph{(Lehmer) code} of $w$: $\code(w):=(c_1,...,c_{n-1})$ where $c_i=\# \{j>i \: | \: w_j < w_i\}$.

\begin{theorem}
\label{thm:delta}
\begin{equation} \label{eq:Delta}
\Delta \cdot \tilde{\mf{S}}_w = \sum_{w \lessdot_S wu} c(w,wu) \tilde{\mf{S}}_{wu}
\end{equation}
where $c(w,wu)$ is the Manhattan distance between $\code(w)$ and $\code(wu)$.
\end{theorem}
\begin{proof}
Let $e,f,h$ denote the standard generators of the Lie algebra $\mf{sl}_2(\C)$.  It is clear from the classification of irreducible representations for $\mf{sl}_2(\C)$ (see, for example, \cite{Kirillov2008}) that 
\[
\tilde{V} \cong V_{n-1} \otimes V_{n-2} \otimes \cdots \otimes V_1 \otimes V_0
\]
where $V_i \cong \spn_{\C} \{x_{n-i}^{j}y_{n-i}^{i-j} \: | \: 0 \leq j \leq i\}$ is the $(i+1)$-dimensional irreducible representation of $\mf{sl}_2(\C)$, with the actions of $e$ and $f$ given by $\Delta$ and $\nabla$ respectively.  Here $h$ acts by multiplying monomials $x^{\alpha}y^{\beta} \in \tilde{V}$ by the scalar $|\alpha|-|\beta|$.

Identifying $\tilde{V}$ with $\C W_n$ by $\tilde{\mf{S}}_w \mapsto w$, it was shown in \cite{Gaetz2018} that the operators defined by the right-hand-sides of (\ref{eq:nabla}) and (\ref{eq:Delta}), together with the action of $h$ by $h(w)=\left(2 \ell(w) - N\right) w$, determine a representation of $\mf{sl}_2(\C)$.  As an easy corollary of the Jacobson-Morozov Theorem (or as explicitly shown by Proctor \cite{Proctor1982}) the actions of $e$ and $h$ in an $\mf{sl}_2(\C)$-representation uniquely determine the action of $f$.  Therefore the action of $\Delta$ on $\tilde{V}$ must be given as above.
\end{proof}

It is elementary to see that for $w \lessdot_S wt_{ij}$ a covering relation, $\code(w)$ and $\code(wt_{ij})$ differ only in positions $i$ and $j$, and that $c(w,wt_{ij})$ is an odd positive number.

\section{Path counting identities}
\label{sec:path-counting}

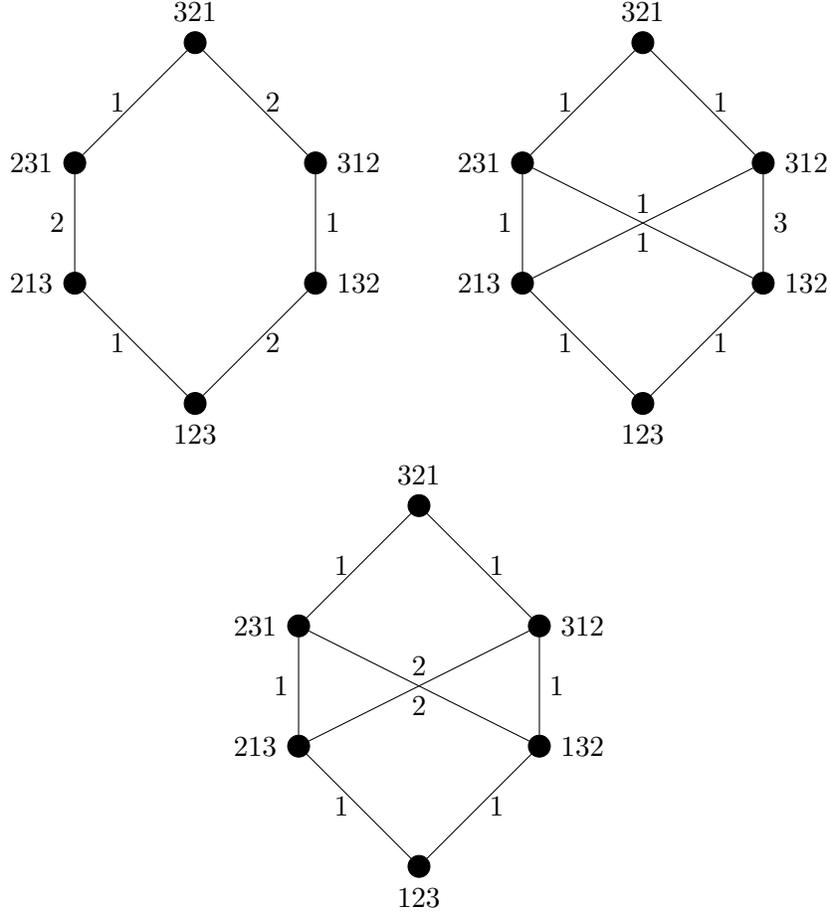
\begin{figure} 
\begin{tikzpicture} [scale=1.6]
%[every node/.style={draw,shape=circle,fill=black,scale=1}, scale=2]
\node[draw,shape=circle,fill=black,scale=0.8] (a0)[label=below:{$123$}] at (0,0) {};
\node[draw,shape=circle,fill=black,scale=0.8](b0)[label=left:{$213$}] at (-1,1) {};
\node[draw,shape=circle,fill=black,scale=0.8](b1)[label=right:{$132$}] at (1,1) {};
\node[draw,shape=circle,fill=black,scale=0.8](c0)[label=left:{$231$}] at (-1,2) {};
\node[draw,shape=circle,fill=black,scale=0.8](c1)[label=right:{$312$}] at (1,2) {};
\node[draw,shape=circle,fill=black,scale=0.8](d0)[label={$321$}] at (0,3) {};

\draw (a0) -- node[left] {$1$} (b0);
\draw (a0) -- node[right] {$2$} (b1);
\draw (b0) -- node[left] {$2$} (c0);
\draw (b1) -- node[right] {$1$} (c1);
\draw (c0) -- node[left] {$1$} (d0);
\draw (c1) -- node[right] {$2$} (d0);
\end{tikzpicture} \hspace{0.2in}
\begin{tikzpicture} [scale=1.6]
%[every node/.style={draw,shape=circle,fill=black,scale=1}, scale=2]
\node[draw,shape=circle,fill=black,scale=0.8] (a0)[label=below:{$123$}] at (0,0) {};
\node[draw,shape=circle,fill=black,scale=0.8](b0)[label=left:{$213$}] at (-1,1) {};
\node[draw,shape=circle,fill=black,scale=0.8](b1)[label=right:{$132$}] at (1,1) {};
\node[draw,shape=circle,fill=black,scale=0.8](c0)[label=left:{$231$}] at (-1,2) {};
\node[draw,shape=circle,fill=black,scale=0.8](c1)[label=right:{$312$}] at (1,2) {};
\node[draw,shape=circle,fill=black,scale=0.8](d0)[label={$321$}] at (0,3) {};

\draw (a0) -- node[left] {$1$} (b0);
\draw (a0) -- node[right] {$1$} (b1);
\draw (b0) -- node[left] {$1$} (c0);
\draw (b1) -- node[right] {$3$} (c1);
\draw (b0) -- node[below] {$1$} (c1);
\draw (b1) -- node[above] {$1$} (c0);
\draw (c0) -- node[left] {$1$} (d0);
\draw (c1) -- node[right] {$1$} (d0);
\end{tikzpicture} \hspace{0.2in}
\begin{tikzpicture} [scale=1.6]
%[every node/.style={draw,shape=circle,fill=black,scale=1}, scale=2]
\node[draw,shape=circle,fill=black,scale=0.8] (a0)[label=below:{$123$}] at (0,0) {};
\node[draw,shape=circle,fill=black,scale=0.8](b0)[label=left:{$213$}] at (-1,1) {};
\node[draw,shape=circle,fill=black,scale=0.8](b1)[label=right:{$132$}] at (1,1) {};
\node[draw,shape=circle,fill=black,scale=0.8](c0)[label=left:{$231$}] at (-1,2) {};
\node[draw,shape=circle,fill=black,scale=0.8](c1)[label=right:{$312$}] at (1,2) {};
\node[draw,shape=circle,fill=black,scale=0.8](d0)[label={$321$}] at (0,3) {};

\draw (a0) -- node[left] {$1$} (b0);
\draw (a0) -- node[right] {$1$} (b1);
\draw (b0) -- node[left] {$1$} (c0);
\draw (b1) -- node[right] {$1$} (c1);
\draw (b0) -- node[below] {$2$} (c1);
\draw (b1) -- node[above] {$2$} (c0);
\draw (c0) -- node[left] {$1$} (d0);
\draw (c1) -- node[right] {$1$} (d0);
\end{tikzpicture}
\caption{The weak order weights (top left), code weights (top right), and Chevalley weights (bottom) for the symmetric group $S_3$.  By Proposition \ref{prop:chev-path-counting} and Theorem \ref{thm:path-counting}, in all cases there are $6={3 \choose 2}!$ weighted paths from 123 to 321.}
\label{fig:weights}
\end{figure}

We call the strong order edge weights $\wt_{\Delta}(w,wt_{ij})$ given by (\ref{eq:Delta}) the \emph{code weights}.  These are different from the previously studied \emph{Chevalley weights}: 
\[
\wt_{Chev}(w,wt_{ij}):=j-i.
\]
The weak order weights $\wt_{\nabla}$, and the two strong order weights are shown in Figure \ref{fig:weights}. 

We now observe a symmetry possessed by all weight functions under consideration.  This symmetry corresponds to the symmetry of the weighted posets in Figure \ref{fig:weights} given by reflecting about a horizontal line.

\begin{prop} \label{prop:w0-symmetry}
Let $\wt_{\ph}$ be one of $\wt_{\nabla}, \wt_{\Delta},$ or $\wt_{Chev}$, and let $\leq$ denote the corresponding order (weak order for $\ph=\nabla$ and strong order for the others).  Suppose $u \lessdot w$ is a covering relation, then $w_0w \lessdot w_0u$ and $\wt_{\ph}(u,w)=\wt_{\ph}(w_0w, w_0u)$.  Therefore for $v \leq w$ we have $m_{\ph}(v,w)=m_{\ph}(w_0w, w_0v)$.
\end{prop}
\begin{proof}
It is straightforward to see that, for $w=w_1,...,w_n$ in one-line notation, we have $w_0w=(n+1)-w_1,...,(n+1)-w_n$; from this it is clear that $w_0w \lessdot w_0u$.  The second claim for $\wt_{\nabla}$ and $\wt_{Chev}$ follows because we swap the same positions to get from $w_0w$ to $w_0u$ as we do to get from $u$ to $w$.  For $\wt_{\Delta}$ this follows immediately from the alternate description for $c(w,wt_{ij})$ given in \cite{Gaetz2018}.  Finally, every chain $v \to w$ corresponds to a chain $w_0w \to w_0v$ with the same edge labels by multiplying all elements by $w_0$, thus the total weighted path counts are the same.
\end{proof}

The following fact is due essentially to Chevalley, and has been further studied by Stembridge \cite{Stembridge2002} and Postnikov and Stanley \cite{Postnikov2009}.

\begin{prop} \label{prop:chev-path-counting}
For the Chevalley weights on the strong order and the $\nabla$-weights on the weak order, we have:
\[
m_{Chev}(\e, w_0)=m_{\nabla}(\e, w_0)=N!.
\]
\end{prop}

Macdonald's celebrated reduced word identity expresses $\mf{S}_u(1,...,1)$ in terms of weighted reduced words in the weak order.

\begin{theorem}[Macdonald \cite{Macdonald}] \label{thm:macdonald-reduced-word}
For any $u \in S_n$:
\[
m_{\nabla}(\e, u) = \ell(u)! \mf{S}_u(1,...,1).
\]
\end{theorem}

Theorem \ref{thm:path-counting} below provides a strong order analog of Theorem \ref{thm:macdonald-reduced-word}.  It also shows that strong order path counting identities with the code weights satisfy the same formula as the Chevalley weights do in Proposition \ref{prop:chev-path-counting}, but also hold for intervals other than $[\e, w_0]$.  See \cite{chain-paper} for a common generalization of the code weights and the Chevalley weights. 

\begin{theorem} \label{thm:path-counting}
For any $u \in S_n$:
\begin{align*}
\frac{m_{\Delta}(u, w_0)}{(N-\ell(u))!} &= \frac{m_{\nabla}(\e,u)}{\ell(u)!} &= \mf{S}_u(1,...,1) &=
\frac{m_{\Delta}(\e,w_0 u)}{(N-\ell(u))!} &= \frac{m_{\nabla}(w_0 u, w_0)}{\ell(u)!}.
\end{align*}
In particular, 
\[
m_{\Delta}(\e, w_0)=N!.
\]
\end{theorem}
\begin{proof}
It was observed in \cite{Hamaker2018} that $\nabla^{|\alpha|} \cdot  x^{\alpha}y^{\rho-\alpha} = |\alpha|! y^{\rho}$.  Similarly, it is clear that $\Delta^{N-|\alpha|} \cdot x^{\alpha}y^{\rho-\alpha} = (N-|\alpha|)! x^{\rho}$.  Applying this to the padded Schubert basis for $\tilde{V}$ yields the first result.  The second result follows immediately from Proposition \ref{prop:w0-symmetry}.
\end{proof}

The simple relation between $m_{\Delta}$ and $m_{\nabla}$ from Theorem \ref{thm:path-counting} does not hold for general pairs of permutations, however in Section \ref{sec:smith-forms} we show a very strong relationship in general between the matrices for $\Delta^{[i,j]}$ and $\nabla^{[N-j, N-i]}$ in the padded Schubert basis (or equivalently, permutation, basis).

\begin{remark}
For \emph{affine} Weyl groups, Lam and Shimozono \cite{Lam} give a dual-graded graph structure to the Hasse diagrams of the weak and strong order.  The relationship, if any exists, between their duality and ours is not currently understood.
\end{remark}

\section{Smith normal forms for $\nabla$ and $\Delta$}
\label{sec:smith-forms}

Stanley \cite{Stanley2017} asked for a description of the Smith normal form of $\nabla^{[\ell,N-\ell]}$, Theorem \ref{thm:snf} below answers this question in greater generality.

\begin{theorem} \label{thm:snf}
For $0 \leq \ell < \ell' \leq N$ and $\ell+\ell' \leq N$, in the padded Schubert basis (or permutation basis) for $\tilde{V}$ (or $\C W_n \cong \C S_n$) we have
\[
\tilde{\snf}{\Delta^{[\ell,\ell']}}=\tilde{\snf}{\Delta^{[N-\ell',N-\ell]}}=\tilde{\snf}{\nabla^{[\ell,\ell']}}=\tilde{\snf}{\nabla^{[N-\ell',N-\ell]}}
\]
and all are equal to
\[
(\ell'-\ell)! \snf(\mc{D})
\]
where $\mc{D}$ is the diagonal matrix with $\#(W_n)_i - \#(W_n)_{i-1}$ entries equal to ${\ell' -i \choose \ell-i}$ for $i=0,1,...,\ell$.
\end{theorem}

\begin{example} \label{ref:main-theorem-example}
According to Figure \ref{fig:weights}, we have 
\begin{align*}
    &\Delta^{[1,2]} = \begin{pmatrix} 1 & 1 \\ 1 & 3 \end{pmatrix}
    &\nabla^{[1,2]} = \begin{pmatrix} 2 & 0 \\ 0 & 1 \end{pmatrix}
\end{align*}
It is easy to check that both of these matrices have Smith normal form $\begin{pmatrix} 1 & 0 \\ 0 & 2 \end{pmatrix}$.  Even in this small example it is clear that the Chevalley weights do not have this property.
\end{example}

% new stuff starts here

We will prove Theorem \ref{thm:snf} by studying Smith normal forms of raising operators in products of chains.

If $P$ and $Q$ are two rank-symmetric and rank-unimodal posets, with order raising operators $U_P$ and $U_Q$, their Cartesian product $P\times Q$ is also rank-symmetric and rank-unimodal. Since each covering relation in $P \times Q$ comes from a covering relation in $P$ or in $Q$, we can define an order raising operator $U_{P\times Q}$ which inherits the edge weights of $U_P$ and $U_Q$.

By a chain of length $M_1$ we mean a totally ordered set of $M_1+1$ elements, which we often identify with monomials $1,x,...,x^{M_1}$. Unless otherwise noted, such a chain has standard raising operator $U: x^a \mapsto (a+1)x^{a+1}$.  Let $P^{(M_1,\ldots,M_k)}$ be the poset obtained by taking the product of chains of length $M_1,\ldots,M_k$; elements in the $n$-th rank are labeled by monomials of degree $n$ in $x_1,...,x_k$.  We write $M$ for $(M_1,\ldots,M_k)$, $|M|$ for $M_1+\cdots+M_k$ and $\alpha\leq M$ for the condition that $\alpha_i\leq M_i$ for all $i$, so $P_n^M=\{x^{\alpha}\ |\ 0\leq\alpha\leq M, |\alpha|=n\}.$  We denote the raising operator inherited from the standard raising operators on the individual chains by $U_M$.  Explicitly, $U_M$ acts as the transpose of $\nabla$:
$$U_M(x^\alpha)=\sum_{i=1}^k(\alpha_i+1)x_ix^{\alpha}$$
where the terms $x^{\beta}$ with $\beta\nleq M$ are considered to be 0. We also have
$$\frac{U_M^n}{n!}x^{\alpha}=\sum_{\alpha\leq\beta\leq M,|\beta|=|\alpha|+n}{\beta\choose\alpha}x^{\beta}$$
where ${\beta\choose\alpha}={\beta_1\choose\alpha_1}\cdots{\beta_k\choose\alpha_k}.$ To see this, note that the coefficient in front of $x^{\beta}$ after applying $U_M^n$ to $x^{\alpha}$ can be viewed as a weighted sum of lattice paths from $\alpha$ to $\beta$, where each of the ${|\beta|-|\alpha|\choose{\beta_1-\alpha_1,\ldots,\beta_k-\alpha_k}}$ paths has weight $\big((\alpha_1+1)\cdots(\beta_1)\big)\cdots\big((\alpha_k+1)\cdots(\beta_k)\big)$.

Lemma \ref{lem:base-change-main} introduces a new integer basis $B_n^M$ for $\Z P_n^M$ with respect to which the Smith normal form of $U_M$ can more easily be computed. 

\begin{lemma}\label{lem:base-change-main}
Fix $M=(M_1,\ldots,M_k)$. For each $0\leq n\leq |M|/2$, there exists a subset $A_n^M\subset P_n^M$ of cardinality $\#P_n^M-\#P_{n-1}^M$, such that
$$B_n^M:=\bigcup_{m\leq n}\left\{\frac{U_M^{n-m}}{(n-m)!}f\ \Big|\ f\in A_m^M\right\}$$
is an integral basis for $\Z P_n^M$.
\end{lemma}

As a product of chains, $P^M$ is rank-symmetric and rank-unimodal, so for $n\leq|M|/2$ the quantity $\# P_n^M-\# P_{n-1}^M$ is always nonnegative. As an example, when $k=1$ (so $P^M$ is a chain) we have $A_0^M=\{1\}$ and $A_n^M=\emptyset$ for all $n\geq1$, thus we recover the monomial basis $\frac{U^n}{n!}\cdot1=x_1^n$. Before proving Lemma~\ref{lem:base-change-main}, we notice some of its consequences. 
\begin{lemma}\label{lem:base-change-higher}
Fix $M=(M_1,\ldots,M_k)$. Suppose that for all $0\leq n\leq |M|/2$, there exists $A_n^M\subset P_n^M$ satisfying the conditions in Lemma~\ref{lem:base-change-main}, then for $n'\geq|M|/2$,
$$B_{n'}^M:=\bigcup_{m\leq |M|-n'}\left\{\frac{U^{n'-m}_M}{(n'-m)!}f\ \Big|\ f\in A_m^M\right\}$$
is an integral basis for $\Z P_{n'}^M.$
\end{lemma}
\begin{proof}
Fix $n'\geq|M|/2$ and let $n=|M|-n' \leq|M|/2$. There is a natural lowering operator $D_M$ such that $[U_M,D_m]=H_M$ and $U_M,D_M,H_M$ together make $\C P_M$ an $\mathfrak{sl}_2$ representation, where $H_M$ acts diagonally. Specifically, 
$$D_M(x^{\beta})=\sum_{i=1}^k(M_i-\beta_i+1)x^{\beta}/x_i$$
where the terms with negative components are considered to be 0. Comparing $U_M$ with $D_M$, we see that, as matrices, $U_M^{[\ell,\ell']}$ and $D_M^{[|M|-\ell',|M|-\ell]}$ are transposes of each other. In particular, when going between rank $n$ and rank $n'$, $|\det U_M^{[n,n']}|=|\det D_M^{[n,n']}|.$

At the same time, by decomposing $\C P^M$ into irreducible representations, we observe that 
\begin{align*}
    \det\left(U_M^{[n,n']}\right)\det\left(D_M^{[n,n']}\right)&=\det \left( U_M^{[n,n']}D_M^{[n,n']} \right)
    \\ &=\prod_{0\leq m\leq n}\left(\frac{(n'-m)!}{(n-m)!}\right)^{2(\# P_m^M-\# P_{m-1}^M)}.
\end{align*}
Here we have benefited from the fact that $U_M^{[n,n']}D_M^{[n,n']}$ is a map from a vector space to itself, so its determinant is independent of the basis used.  Therefore, we obtain:
$$\big|\det\left(U_M^{[n,n']}\right)\big|=\prod_{0\leq m\leq n}\left(\frac{(n'-m)!}{(n-m)!}\right)^{\# P_m^M-\# P_{m-1}^M}.$$

We now apply Lemma~\ref{lem:base-change-main}. Consider the linear map $\Z B_{n}^M\rightarrow\Z P_{n}^M\rightarrow\Z P_{n'}^M$ that is a change of basis $V_n^{-1}$ from $B_n^M$ to $P_n^M$ composed with $U_M^{[n,n']}$.  Notice that since both $B_n^M$ and $P_n^M$ are integral basis for $\Z P_n^M$, the change of basis matrix $V_n$ and $V_n^{-1}$ both have integer entries and determinant $\pm1$. Another way to write this linear map is through $\Z B_n^M\rightarrow\Z B_{n'}^M\rightarrow\Z P_{n'}^M$ and let us write it as $V_{n'}R_M^{[n,n']}$. Here, $V_{n'}$ represents how each $\frac{U_M^{n'-m}}{(n'-m)!}f$, $f\in A_m^M\subset P_m^M$ is written in the basis $P_{n'}^M$ so it has integer entries (which are products of binomial coefficients). In the meantime, by definition of $B_n^M$ and $B_{n'}^M$, $R_M^{[n,n']}$ is a diagonal matrix with $\# P_m^M-\# P_{m-1}^M$ copies of $(n'-m)!/(n-m)!$ on the diagonal. This shows $\det U_M^{[n,n']}=\det R_M^{[n,n']}$. With $U_M^{[n,n']}V_n^{-1}=V_{n'}R_M^{[n,n']}$, we have $\det V_{n'}=\pm1$ and this precisely means that $B_{n'}^M$ is an integral basis of $\Z P_{n'}^M.$
\end{proof}

We now prove Lemma~\ref{lem:base-change-main}.

\begin{proof}[Proof of Lemma~\ref{lem:base-change-main}]
We use induction on $k$ (the number of chains), $|M|$ and then $n$, in this order. When $k=1$, let $A_0^M=\{1\}$ and $A_m^M=\emptyset$ for $m\geq1$. Then $B_n^M=P_n^M$ is an integral basis for $\Z P^M_n$, as desired. 

Now assume $k\geq2$ and also assume without loss of generality that $M_1\geq M_2\geq\cdots\geq M_k\geq1$, since when $M_k=0$, we reduce to the case of $k-1$. When $n=0$, choose $A_0^M=\{1\}$. Assume $1\leq n\leq |M|/2$ and by the induction hypothesis assume that $A_m^M$ has already been chosen with the desired properties for all $m<n$. Recall that for any $f\in P_m^M$, $U_M^{n-m}f/(n-m)!$ is an integral linear combination of $P_n^M$. By simple counting of dimensions, it suffices to show that we can choose $A_n^M$ with cardinality $\# P_n^M-\# P_{n-1}^M$ such that every element in $P_n^M$ can be written as an integral linear combination of 
$$\bigcup_{m<n}\left\{\frac{U_M^{n-m}}{(n-m)!}f\ \Big|\ f\in A_m^M\right\}\cup A_n^M.$$

By the induction hypothesis, any $f\in P_m^M$ with $m<n$ can be written as a linear combination
$$f=\sum_{m'<m,\ g\in A_{m'}^M}c_g\frac{U_M^{m-m'}}{(m-m')!}g$$
with $c_g \in \Z$.  Applying $U_M^{n-m}/(n-m)!$ on both sides, we obtain
$$\frac{U_M^{n-m}}{(n-m)!}f=\sum_{m'<m,\ g\in A_{m'}^M}c_g{n-m'\choose n-m}\frac{U_M^{n-m'}}{(n-m')!}g.$$
Hence it suffices to find $A_n^M\subset P_n^M$ and show that every element in $P_n^M$ is an integral linear combination of $\overline{B_n^M}\cup A_n^M$, where
$$\overline{B_n^M}:=\bigcup_{m<n}\left\{\frac{U_M^{n-m}}{(n-m)!}f\ \Big|\ f\in P_m^M\right\}$$
is larger than $P_{n-1}^M$. This step of reduction allows us to not be concerned with the choices of $A_m^M$ for $m<n$, and focus only on $A_n^M.$

If $n<M_k$, we reduce to the case $M'=(M_1,M_2,\ldots,M_{k-1},n)$ (since $P^M$ and $P^{M'}$ are isomorphic below rank $n$), and let $A_n^M=A_n^{M'}$. The key case is the range $M_k\leq n\leq |M|/2$. Let
$$A_n^M=A_n^{(M_1,\ldots,M_{k-1},M_k-1)}\cup A_{n-M_k}^{(M_1,\ldots,M_{k-1})}\cdot x_k^{M_k}$$
where the second terms means $\{f\cdot x_k^{M_k}\ |\ f\in A_{n-M_k}^{(M_1,\ldots,M_{k-1})}\}.$ We always have $n-M_k\leq (M_1+\cdots+M_{k-1})/2$ but it is possible that $(|M|-1)/2<n \leq |M|/2$, which implies that $2n=|M|$. For simplicity, write $M'=(M_1,\ldots,M_{k-1},M_k-1)$ and $M''=(M_1,\ldots,M_{k-1})$. Notice that the entries of both $M'$ and $M''$ are still decreasing. We need to show that $A_n^M$ has the desired cardinality and that integral linear combinations of $\overline{B_n^M}\cup A_n^M$ generate $P_n^M$.

For any $0\leq\alpha\leq M$, depending on whether $\alpha_k=M_k$ or not, we can partition $P_n^M$ into two subsets $P_n^{M'}\cup P_{n-M_k}^{M''}\cdot x_k^{M_k}$, with the convention that $P_{<0}^{M''}=\emptyset$. This means $\# P_n^M=\# P_n^{M'}+\# P_{n-M_k}^{M''}$. Then
$$\# P_n^M-\#P_{n-1}^M=(\# P_n^{M'}-\# P_{n-1}^{M'})+(\#P_{n-M_k}^{M''}-\#P_{n-M_k-1}^{M''}).$$ By the induction hypothesis, $\#P_{n-M_k}^{M''}-\#P_{n-M_k-1}^{M''}$ is the cardinality of $A_{n-M_k}^{M''}\cdot x_k^{N_k}$ and $\# P_n^{M'}-\# P_{n-1}^{M'}$ is the cardinality of $A_n^{M'}$ when $n\leq |M'|/2$. When $|M|=2n$, we have $|M'|=2n-1$ and so ranks $n-1$ and $n$ are both in the middle of the rank-symmetric poset $P^{M''}$, meaning $\# P_n^{M'}-\# P_{n-1}^{M'}=0$. In this case we let $A_n^{M'}=\emptyset$ and thus $\# P_n^{M'}-\# P_{n-1}^{M'}$ is indeed $\# A_n^{M'}$ whenever $n\leq |M|/2$. Therefore, $\#A_n^M=\# P_n^M-\# P_{n-1}^M$ as desired.

To show $\overline{B_n^M}$ and $A_n^M$ generate, let us first deal with the monomials in $P_n^M$ that are multiples of $x_k^{M_k}$. Notice that the action of $U_{M''}$ on $h\in P_{m-M_k}^{M''}$ is exactly the same as the action of $U_M$ on $hx_k^{M_k}\in P_m^{M}$ since the exponents of $x_k$ can never go up. By the induction hypothesis on $M''$, $\overline{B_{n-M_k}^{M''}}\cup A_{n-M_k}^{M''}$ can generate $h\in P_{n-M_k}^{M''}$ so $\overline{B_{n-M_k}^{M''}}\cdot x_k^{N_k}\cup A_{n-M_k}^{M''}\cdot x_k^{N_k}\subset\overline{B_n^M}\cup A_n^M$ can generate $hx_k^{N_k}\in P_n^M$. 

Next, we deal with monomials that are not multiples of $x_k^{M_k}$. In other words, these are the monomials in $P_n^{M'}$. We claim that each $h\in P_n^{M'}$ can be written as an integral linear combination of $\overline{B_n^{M'}}\cup A_n^{M'}$, where the definition of $\overline{B_n^{M'}}$ (introduced earlier in this proof) remains the same for $n>|M'|/2$. To see this, when $n\leq |M'|/2$, it is simply the induction hypothesis, and when $n=|M|/2$, $A_n^{M'}=\emptyset$ and the claim follows from Lemma~\ref{lem:base-change-higher} as $B_n^{M'}\subset\overline{B_n^{M'}}$. Recall that the goal is to show that $h\in P_n^{M'}$ is an integral linear combination of $\overline{B_n^{M}}\cup A_n^{M}$ and so far we know that it is an integral linear combination of $\overline{B_n^{M'}}\cup A_n^{M'}$. Correspondingly, write
$$h=\sum_{m<n,\ g\in P_m^{M'}}c_g\cdot\frac{U_{M'}^{n-m}}{(n-m)!}\cdot g+\sum_{f\in A_n^{M'}}c_f\cdot f.$$
Using the same coefficients, consider
$$h'=\sum_{m<n,\ g\in P_m^{M'}\subset P_m^M}c_g\cdot\frac{U_{M}^{n-m}}{(n-m)!}\cdot g+\sum_{f\in A_n^{M'}\subset A_n^M}c_f\cdot f,$$
which lies in the integral span of $\overline{B_n^{M}}\cup A_n^{M}$. The difference between $h'$ and $h$ is an integral linear combination of 
$$\frac{U_M^{n-m}}{(n-m)!}\cdot x^{\alpha}-\frac{U_{M'}^{n-m}}{(n-m)!}x^{\alpha}=\sum_{\alpha\leq\beta\leq M,\ \beta_k=M_k,\ |\beta|=n}{\beta\choose\alpha}x^{\beta}$$
for $x^{\alpha}\in P_m^{M'}$. As we already know that $x^{\beta}$ for $\beta_k=M_k$ lies in the integral span of $\overline{B_n^{M}}\cup A_n^{M}$ (see last paragraph), we conclude that any $h\in P_n^M$ lies in this span as well. This finishes the proof.
\end{proof}

The above base change procedure allows easy computation of the Smith normal form of powers of $U_M$.

\begin{lemma}\label{lem:product-snf}
For $M=(M_1,\ldots,M_k)$, let $U_M$ be the standard order raising operator on the product of chains $P^M$. For $0\leq\ell<\ell'\leq |M|$ and $\ell+\ell'\leq |M|$, we have
\[
\tilde{\snf}\left(U_M^{[\ell,\ell']}\right)=\tilde{\snf}\left(U_M^{[|M|-\ell',|M|-\ell]}\right)=(\ell'-\ell)!\snf(\mc{D})
\]
where $\mc{D}$ is the diagonal matrix with $\# P^M_i - \#P^M_{i-1}$ entries equal to ${\ell' -i \choose \ell-i}$ for $i=0,1,\ldots,\ell$.
\end{lemma}
\begin{proof}
As matrices, $U_M^{[\ell,\ell']}$ and $U_M^{[|M|-\ell',|M|-\ell]}$ are transposes of each other, so it suffices to consider $U_P^{[\ell,\ell']}$.

Instead of the monomial basis $P_n^M$ for $\Z P_n^M$, we use the basis $B_n^M$ described in Lemma~\ref{lem:base-change-main} for $n\leq|M|/2$ and Lemma~\ref{lem:base-change-higher} for $n\geq|M|/2$ (See Figure~\ref{fig:labeled-chain-example} for a visualization of the order raising operator $U_M$ in this new basis). The smith normal form of $U_M^{[\ell,\ell']}$ remains unchanged because the change of basis is integral. In this new basis, as long as $\ell+\ell'\leq|M|$, we see from the definition of $B_n^M$'s that $U_M^{[\ell,\ell']}$ becomes diagonal in the maximal square submatrix and 0 elsewhere. The diagonal consists of $\# A_i^P=\# P_i^M-\# P_{i-1}^M$ copies of $(\ell'-i)!/(\ell-i)!=(\ell'-\ell)!{{\ell'-i}\choose{\ell-i}}$.
\end{proof}

\begin{figure}[h!]
\begin{tikzpicture} 
%[every node/.style={draw,shape=circle,fill=black,scale=0.6}, scale=1.2]
\node[draw,shape=circle,fill=black,scale=0.6](a0) at (0,0) {};
\node[draw,shape=circle,fill=black,scale=0.6](a1) at (0,1) {};
\node[draw,shape=circle,fill=black,scale=0.6](a2) at (0,2) {};
\node[draw,shape=circle,fill=black,scale=0.6](a3) at (0,3) {};
\node[draw,shape=circle,fill=black,scale=0.6](a4) at (0,4) {};
\node[draw,shape=circle,fill=black,scale=0.6](a5) at (0,5) {};
\node[draw,shape=circle,fill=black,scale=0.6](a6) at (0,6) {};

\draw (a0) -- node[left] {$1$} (a1);
\draw (a1) -- node[left] {$2$} (a2);
\draw (a2) -- node[left] {$3$} (a3);
\draw (a3) -- node[left] {$4$} (a4);
\draw (a4) -- node[left] {$5$} (a5);
\draw (a5) -- node[left] {$6$} (a6);

\node[draw,shape=circle,fill=black,scale=0.6](b1) at (1,1) {};
\node[draw,shape=circle,fill=black,scale=0.6](b2) at (1,2) {};
\node[draw,shape=circle,fill=black,scale=0.6](b3) at (1,3) {};
\node[draw,shape=circle,fill=black,scale=0.6](b4) at (1,4) {};
\node[draw,shape=circle,fill=black,scale=0.6](b5) at (1,5) {};

\draw (b1) -- node[left] {$1$} (b2);
\draw (b2) -- node[left] {$2$} (b3);
\draw (b3) -- node[left] {$3$} (b4);
\draw (b4) -- node[left] {$4$} (b5);

\node[draw,shape=circle,fill=black,scale=0.6](c1) at (-1,1) {};
\node[draw,shape=circle,fill=black,scale=0.6](c2) at (-1,2) {};
\node[draw,shape=circle,fill=black,scale=0.6](c3) at (-1,3) {};
\node[draw,shape=circle,fill=black,scale=0.6](c4) at (-1,4) {};
\node[draw,shape=circle,fill=black,scale=0.6](c5) at (-1,5) {};

\draw (c1) -- node[left] {$1$} (c2);
\draw (c2) -- node[left] {$2$} (c3);
\draw (c3) -- node[left] {$3$} (c4);
\draw (c4) -- node[left] {$4$} (c5);

\node[draw,shape=circle,fill=black,scale=0.6](d2) at (2,2) {};
\node[draw,shape=circle,fill=black,scale=0.6](d3) at (2,3) {};
\node[draw,shape=circle,fill=black,scale=0.6](d4) at (2,4) {};

\draw (d2) -- node[left] {$1$} (d3);
\draw (d3) -- node[left] {$2$} (d4);

\node[draw,shape=circle,fill=black,scale=0.6](e2) at (-2,2) {};
\node[draw,shape=circle,fill=black,scale=0.6](e3) at (-2,3) {};
\node[draw,shape=circle,fill=black,scale=0.6](e4) at (-2,4) {};

\draw (e2) -- node[left] {$1$} (e3);
\draw (e3) -- node[left] {$2$} (e4);

\node[draw,shape=circle,fill=black,scale=0.6](f3) at (3,3) {};

\draw[dashed] (f3)--(a4);
\draw[dashed] (f3)--(b4);
\draw[dashed] (f3)--(c4);
\draw[dashed] (f3)--(d4);
\draw[dashed] (f3)--(e4);

\draw[dashed] (d4)--(a5);
\draw[dashed] (d4)--(b5);
\draw[dashed] (d4)--(c5);
\draw[dashed] (e4)--(a5);
\draw[dashed] (e4)--(b5);
\draw[dashed] (e4)--(c5);

\draw[dashed] (b5)--(a6);
\draw[dashed] (c5)--(a6);

\end{tikzpicture}
\caption{The order raising operator $U_M$ in the basis $B_n^M$, where $M=[3,2,1]$.  The dashed edges have weights which we do not describe explicitly, but they are immaterial to $U_M^{[\ell,\ell']}$ when $\ell+\ell' \leq |M|$.}
\label{fig:labeled-chain-example}
\end{figure}
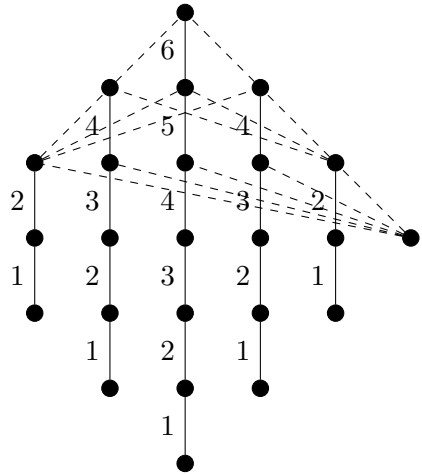

\begin{proof}[Proof of Theorem \ref{thm:snf}]
Fix $0\leq \ell<\ell'\leq N$ with $\ell+\ell'\leq N$. According to Proposition~\ref{prop:w0-symmetry}, in the padded Schubert basis the matrices $\Delta^{[\ell,\ell']}$ and $\Delta^{[N-\ell',N-\ell]}$ are transposes. Therefore, $\tilde\snf\Delta^{[\ell,\ell']}=\tilde\snf\Delta^{[N-\ell',N-\ell]}$. Similarly, $\tilde\snf\nabla^{[\ell,\ell']}=\tilde\snf\nabla^{[N-\ell',N-\ell]}$.  By Corollary 6 of \cite{Hamaker2018}, the change of basis matrix between the padded Schubert and monomial bases for $\tilde{V}$ lies in $GL(\mathbb{Z})$, and thus Smith normal forms are the same in either basis.  Viewing $\Delta^{[N-\ell',N-\ell]}$ and $\nabla^{[\ell,\ell']}$ in the monomial basis, they are transposes, and thus have the same $\tilde\snf$. Finally, in the monomial basis, their Smith normal forms are those of $U_M^{[\ell,\ell']}$ with $M=(n-1,n-2,\ldots,1,0)$. Applying Lemma~\ref{lem:product-snf} gives us the desired result.
\end{proof}

\section*{Acknowledgements}
The authors wish to thank Alex Postnikov and Richard Stanley for helpful conversations.

\bibliographystyle{plain}
\bibliography{arxiv-v2}

\end{document}